\documentclass[12pt]{amsart}
\usepackage{a4wide}
\usepackage[utf8]{inputenc}
\usepackage[T1]{fontenc} 
\usepackage[english]{babel}
\usepackage{amsmath}
\usepackage{amsfonts}
\usepackage{amssymb}
\usepackage{amsthm}
\usepackage{graphicx}
\usepackage{subcaption}
\usepackage{enumerate}
\usepackage{color}
\usepackage{todonotes}
\usepackage{bbm}
\usepackage{verbatim}%for comments

\newcommand{\eps}{\varepsilon}

\renewcommand{\d}{\,\mathrm{d}}
\newcommand{\E}[1]{\mathbb{E}\left[#1\right]} % \E{X}
\newcommand{\no}[1]{\Vert#1\Vert} % norm
\newcommand{\nos}[1]{\Vert#1\Vert^2} % norm square
\newcommand{\be}[1]{\vert#1\vert} % absolut value
\newcommand{\bes}[1]{\vert#1\vert^2} % absolut value  square
\renewcommand{\E}[1]{\mathbb{E}\left[#1\right]} % \E{X}
\renewcommand{\no}[1]{\Vert#1\Vert} % norm
\renewcommand{\nos}[1]{\Vert#1\Vert^2} % norm square
\renewcommand{\be}[1]{\vert#1\vert} % absolut value
\renewcommand{\bes}[1]{\vert#1\vert^2} % absolut value

\newcommand{\fe}[1]{\frac{\nabla #1}{ \sqrt{ \vert \nabla #1 \vert^2 +\eps^2}}} 

\newcommand{\R}{\mathbb{R}}

\newcommand{\N}{\mathbb{N}}

\newcommand{\into}{\int_{\mathcal{O}}}
\newcommand{\ska}[1]{\left( #1 \right)} % Skalarprodukt
\renewcommand{\div}{\mathrm{div}}

\newcommand{\intt}{\int_0^t}

\newcommand{\F}{\mathcal{F}}

\renewcommand{\Xi}{X_{\eps,n,h}^{i}}
\newcommand{\Xii}{X_{\delta,n}^{i}}
\newcommand{\xii}{x_{\eps,n}^i}
\newcommand{\Xmin}{X_{\eps,n,h}^{i-1}}
\newcommand{\Xmi}{X_{\delta,n}^{i-1}}

\renewcommand{\xii}{X_{\eps,h}^{i}} %here maybe h as superscript
\newcommand{\xmi}{X_{\eps,h}^{i-1}} %here maybe h as superscript
\newcommand{\Zi}{Z^{i}_{\eps}}

\newcommand{\Zii}{Z^{i}_{\eps}}
\newcommand{\Zmi}{Z^{i-1}_{\eps}}

\newcommand{\Hz}{\mathbb{H}^1_0}
\renewcommand{\L}{\mathbb{L}^2}
\newcommand{\Hm}{\mathbb{H}^{-1}}
\renewcommand{\phi}{\varphi}
\renewcommand{\F}{\mathcal{F}}
\renewcommand{\P}{\mathbb{P}}

\newcommand{\vh}{\Phi_h}

\renewcommand{\O}{\mathcal{O}}
\newcommand{\D}{\O}

\newcommand{\Xdn}{X^{\delta}_{n}}

\newcommand{\Xe}{X^{\delta}}
\newcommand{\Xee}{X^{\eps}}

\newcommand{\Yc}{\overline{X}_{\tau}^{\delta,n}}

\newcommand{\Ycm}{\overline{X}_{\tau_-}^{\delta,n}}
\newcommand{\Xc}{\overline{X}_{\tau,h}^{\eps,n}}

\renewcommand{\F}{\mathcal{F}}
\renewcommand{\P}{\mathbb{P}}

\renewcommand{\O}{\mathcal{O}}

\newcommand{\Je}{\mathcal{J}_{\eps,\lambda}}

\newcommand{\Jeps}{\mathcal{J}_\eps}
\newcommand{\weak}{\rightharpoonup}

\newtheorem{thms}{Theorem}[section]
\newtheorem{defs}{Definition}[section]
\newtheorem{cors}{Corollary}[section]

\newtheorem{props}{Proposition}[section]
\newtheorem{bems}{Remark}[section]
\newtheorem{lems}{Lemma}[section]

\newcommand{\xx}{\mathbf{x}}

\begin{document}
\title[Numerical approximation of the stochastic TV flow in $d\geq1$]{Convergent numerical approximation of the stochastic total variation flow with linear multiplicative noise: the higher dimensional case}

\author{\v{L}ubom\'{i}r Ba\v{n}as}
\address{Department of Mathematics, Bielefeld University, 33501 Bielefeld, Germany}
\email{banas@math.uni-bielefeld.de}
\author{Michael R\"ockner}
\address{Department of Mathematics, Bielefeld University, 33501 Bielefeld, Germany and Academy of Mathematics and Systems Science, CAS, Beijing}
\email{roeckner@math.uni-bielefeld.de}
\author{Andr\'e Wilke}
\address{Department of Mathematics, Bielefeld University, 33501 Bielefeld, Germany}
 \email{ awilke@math.uni-bielefeld.de}

\thanks{Funded by the Deutsche Forschungsgemeinschaft (DFG, German Research Foundation) – SFB 1283/2 2021 – 317210226.}

\begin{abstract}
We consider fully discrete finite element approximation
of the stochastic total variation flow equation (STVF) with linear multiplicative noise
which was previously proposed in \cite{our_paper}.
Due to lack of a discrete counterpart of stronger a priori estimates in higher spatial dimensions
the original convergence analysis of the numerical scheme was limited to one spatial dimension, cf. \cite{stvf_erratum}.
In this paper we generalize the convergence proof to higher dimensions.
\end{abstract}

\maketitle

\section{Introduction}

We study the convergence of numerical approximation of the stochastic total variation flow (STVF) equation
\begin{align}\label{TVF}
\d X&= \div\left(\frac{\nabla X}{\be{\nabla X}}\right) \d t -\lambda (X - g) \d t +X\d W, &&\text{in } (0,T)\times \O, \nonumber\\
X & = 0 && \text{on } (0,T)\times \partial \O, \\
X(0)&=x_0 &&\text{in } \O, \nonumber
\end{align}
where $\O \in \R^d $, $d\geq 1$,
is a bounded, convex polyhedral domain, $\lambda \geq 0$, $T>0$ are constants and $x_0,\, g \in \L$.
For simplicity we take $W$ to be a one dimensional real-valued Wiener process. %, generalization to the case of a sufficiently regular trace-class noise is straightforward.

{We adopt the approach from \cite{our_paper} and construct a fully discrete approximation scheme (cf. (\ref{full_eps_TVF_L2_data}) below) of (\ref{TVF}) using a regularization approach.}
Given a regularization parameter $\eps >0$ we consider the following regularized problem 
\begin{align}\label{eps.TVF}
 \d X^{\eps}&=  \div\left(\frac{\nabla X^{\eps}}{\sqrt{|\nabla X^{\eps} |^2+\eps^2}}\right)\d t-\lambda(X^{\eps}-g)\d t+X^{\eps}\d W &&\text{in } (0,T)\times \O, \nonumber\\
 X^{\eps} & = 0 && \text{on } (0,T)\times \partial \O, \\
 X^{\eps}(0)&=x_0 &&\text{in } \O\,.  \nonumber
\end{align}

{Equations (\ref{TVF}), (\ref{eps.TVF}), respectively, admit unique solutions in the sense of stochastic variational inequalities, see \cite{Roeckner_TVF_paper}, \cite{our_paper}, \cite{stvf_weak}. 
Throughout the paper we refer to the solutions of (\ref{TVF}), (\ref{eps.TVF}) as SVI solutions, see Definition~\ref{def_svi} below.
The first numerical approximation of \eqref{TVF} was constructed in \cite{our_paper} and its convergence was shown by considering the full discretization of the regularized problem (\ref{eps.TVF})
as an intermediate step.
%In the limit with respect to the regularization and discretization parameters
%the numerical solution of \cite{our_paper} is shown to converge to the unique SVI solution of (\ref{TVF}).
The convergence proof of the numerical approximation in \cite{our_paper} relies on the discrete counterpart of
a priori estimates in stronger norm (cf. Lemma~\ref{laplace_energy_estimate} below), which are so-far restricted to spatial dimension $d=1$, cf. \cite{stvf_erratum}.}
The recent work \cite{stvf_weak} shows convergence of numerical approximation with a random walk representation of the noise
to probabilistically weak SVI solutions of (\ref{TVF}). The numerical analysis in \cite{stvf_weak} is valid in higher spatial dimensions $d\geq 1$, but does not cover the case of linear multiplicative noise, except for $d=1$.
In this work we show convergence of the numerical approximation of the stochastic total variation flow (\ref{TVF}) with linear multiplicative noise in spatial dimension $d>1$.

The paper is organized as follows. In Section~\ref{sec_not} we introduce the notation and state some auxiliary results.
The existence of a unique SVI solution of the regularized  stochastic TV flow (\ref{eps.TVF})
and its convergence towards a unique SVI solution of (\ref{TVF}) 
is discussed in Section~\ref{sec_exist}.
In Section~\ref{sec_semi_num} we introduce a time semi-discrete numerical scheme for the regularizared problem (\ref{reg.TVF}) below
and show its convergence to the variational solution of (\ref{reg.TVF}) for initial data with higher regularity.
Finally, in Section~\ref{sec_full_num} we show the convergence  of the fully discrete finite element scheme for the regularizared problem (\ref{eps.TVF})
and show its convergence to the SVI solution of (\ref{TVF}). 
%Numerical experiments are presented in Section~\ref{sec_sim}.

\section{Notation and preliminaries}\label{sec_not}

Throughout the paper by $C$ we denote  a generic positive constant that may change from line to line. 
By $\mathbb{L}^p:=L^p(\D)$ for $1\leq p \leq \infty  $ we denote the standard spaces of $p$-th order integrable functions on $\O$;
we use $\no{\cdot}= \no{\cdot}_{\L}$ for the $\L$-norm and $(\cdot,\cdot)=(\cdot,\cdot)_{\L}$ for the $\L$-inner product. 
 For $k,p \in \N$ we denote the usual Sobolev space  on $\O$ by $(\mathbb{W}^{p,k},\no{\cdot}_{\mathbb{W}^{p,k}})$; for $p=2$ we use $\mathbb{H}^k:=\mathbb{W}^{2,k}$.
Furthermore $\Hz$ stands for the $\mathbb{H}^1$ space with zero trace on $\partial \O$ with its dual denoted as $\Hm$ and
we set $\langle \cdot ,\cdot \rangle=\langle \cdot ,\cdot \rangle_{\Hm \times \Hz}$, where $\langle \cdot ,\cdot \rangle_{\Hm \times \Hz}$ 
is the duality pairing between $\Hz$ and $\Hm$.  

For $u\in \Hz$ we consider the energy functional 
\begin{align*}%\label{def_jepslam}
\Je(u)= \into \sqrt{\bes{\nabla u} +\eps^2} \d \xx + \frac{\lambda}{2} \into \bes{u-g} \d \xx\,,
\end{align*}  
With a slight abuse of notation we set $\mathcal{J}_\eps := \mathcal{J}_{\eps,0}$ if $\lambda=0$ and $\mathcal{J}_\lambda := \mathcal{J}_{0,\lambda}$ if $\eps=0$.

Next, we state basic definitions related to the functions of bounded variation.
\begin{defs}\label{Bounded Variation} 
 A function $u \in L^1(\O)$ is called a function of bounded variation, if its total variation
 \begin{align*}%\label{total variation}
 \into \be{\nabla u}\d\xx  := \sup\left\{-\into u\, \div\, \mathbf{v} \d\xx;~ \mathbf{v} \in C^{\infty}_0(\O,\R^d), ~\no{\mathbf{v}}_{L^{\infty}}\leq 1\right\},
 \end{align*}
 is finite. The space of functions of bounded variation is denoted by $BV(\O)$.

Furthermore, for $u \in BV(\O)$ we set
$$
 \into \sqrt{\be{\nabla u}^2 + \eps^2}\d\xx  := \sup\left\{\into \Big(-u\, \div\, \mathbf{v} + \eps\sqrt{1-|\mathbf{v}|^2}\Big)\d\xx;~ \mathbf{v} \in C^{\infty}_0(\O,\R^d), ~\no{\mathbf{v}}_{L^{\infty}}\leq 1\right\}\,.
$$
\end{defs}

\section{The continuous problem}\label{sec_exist}

In this section we construct a unique SVI solution of (\ref{TVF}) (see Definition~\ref{def_svi} below)
via a two-level regularization procedure. 
{Given the data $x_0 \in L^2(\Omega,\F_0;\L)$, $g\in \L$
we consider an $\Hz$-approximating sequences {$\{x^n_0\}_{n\in \mathbb{N}} \subset L^2(\Omega,\F_0;\Hz)$, $\{g^n\}_{n\in \mathbb{N}} \subset \Hz$,
s.t.  $x^n_0 \rightarrow x_0$, $g^n \rightarrow g$ in $L^2(\Omega,\F_0; \L)$ for $n\rightarrow \infty$, respectively.} 
For $\delta>0$ we  introduce a regularization of (\ref{eps.TVF}) as
\begin{align}\label{reg.TVF}
 \d \Xdn=& \delta \Delta \Xdn + \div\left(\frac{\nabla \Xdn}{\sqrt{|\nabla \Xdn |^2+\eps^2}}\right)\d t 
-\lambda(\Xdn-g_n)\d t+\Xdn\d W(t),\\
 \nonumber \Xdn(0)=&x^n_0\,.
\end{align} 
}
%%%%%%%%%%%%%%%%%%%%%%%%

% To be able to treat problems with $\L$-regular data, i.e., $x_0 \in L^2(\Omega,\F_0;\L)$, $g\in \L$
%we consider a $\Hz$-approximating sequence {$\{x^n_0\}_{n\in \mathbb{N}} \subset L^2(\Omega,\F_0;\Hz)$
%s.t.  $x^n_0 \rightarrow x_0$
%in $L^2(\Omega,\F_0; \L)$ for $n\rightarrow \infty$ and $\{g_n\}_{n\in \mathbb{N}} \subset \Hz$
%s.t.  $g_n \rightarrow g$
%in $\L$ for $n\rightarrow \infty$.}
%Hence, we consider the following regularization of (\ref{reg.TVF}):
%\begin{align}\label{vis.TVF}
%\nonumber \d \Xdn=& \delta \Delta \Xdn + \div\left(\frac{\nabla \Xdn}{\sqrt{|\nabla \Xdn |^2+\eps^2}}\right)\d t 
% \\
% &-\lambda(\Xdn-g_n)\d t+\Xdn\d W(t)  &&\text{in }(0,T)\times \O,\\
% \nonumber \Xdn(0)=&x^n_0 &&\text{in } \O\nonumber.
%\end{align}
We define the operator $A^{\delta} : \Hz \rightarrow \Hm$ as
\begin{align}\label{Operator}
{\langle A^{\delta} u, v\rangle_{\Hm \times \Hz}=  \delta \left(\nabla u,\nabla v\right)  + \left(\fe{u},\nabla v\right) +\lambda \left(u-   g_n,v\right)   ~~\forall  u,v \in \Hz },
\end{align}
and note that \eqref{reg.TVF} can be equivalently formulated as
\begin{align}
\d \Xdn +A^{\delta} \Xdn \d t&=\Xdn \d W(t)\,,\\
\Xdn(0)&=x^n_0. \nonumber
\end{align}

The operator $A^{\delta} : \Hz \rightarrow \Hm$  is demicontinuos and satisfies (cf. \cite[Remark 4.1.1]{Roeckner_book})
\begin{align}
&\langle A^{\delta}(u)-A^{\delta}(v),u-v\rangle_{\Hm \times \Hz}\geq  \delta\nos{u-v}_{\Hz} + \lambda\nos{u-v}, && \forall u,v \in \Hz,\label{Monotonicity}\\
& \no{A^{\delta}(u)}_{\Hm} \leq C(\delta, \lambda,\|g_n\|)(\no{u}_{\Hz}+1), && \forall u \in \Hz.\label{a_bnd}
\end{align}
We recall that the convexity of the function $\sqrt{\bes{\cdot}+\eps^2}$ implies the monotonicity property
\begin{align}\label{eps.convexity.inequality}
&\ska{\fe{X}-\fe{Y},\nabla(X-Y)}
\nonumber\\
&\qquad  = \ska{\fe{X},\nabla(X-Y)}+\ska{\fe{Y},\nabla(Y-X)}
\\
& \qquad \geq  \Jeps(X)-\Jeps(Y)+\Jeps(Y)-\Jeps(X)=0.
\nonumber
\end{align}

The well-posedness of the regularized problem \eqref{reg.TVF} follows from standard theory of monotone SPDEs, see for instance \cite[Chapter 4]{Roeckner_book} and \cite{our_paper}.
\begin{lems}\label{lem_eps_exist}
For any $\eps,\delta\,>0$ and $x^n_0 \in L^2(\Omega,\F_0;\Hz)$,  $g_n\in \Hz$ 
there exists a unique variational solution $\Xe_n \in L^2(\Omega;C([0,T];\L))$ 
of \eqref{reg.TVF}. Furthermore, there exists a constant $C\equiv C(T)>0$ such that the following estimate holds
\begin{align*}
\E{\sup_{t \in [0,T]} \nos{\Xe_n(t)}} \leq C (\E{\nos{x_0^n}} + \|g_n\|^2).
\end{align*}
\end{lems}
We recall that in addition to the above $\L$-estimate, the solution of the regularized equation \eqref{reg.TVF} satisfies the following stronger a priori estimate, see \cite[Lemma 3.2]{our_paper}.
\begin{lems}\label{laplace_energy_estimate}
Let $x_0^n\, \in L^2(\Omega,\F_0;\Hz)$, $g_n \in \Hz$.
There exists a constant $C\equiv C(T)$
such that for any $\eps,\delta\,>0$ the corresponding variational solution $\Xdn$ of \eqref{reg.TVF}  satisfies
\begin{align}
\E{\sup_{t \in [0,T]} \nos{\nabla \Xdn(t)}  + \delta \int_0^T \nos{\Delta \Xdn(t)}\d t  } \leq C\left(\E{\nos{x_0^n}_{\Hz}}+\nos{g_n}_{\Hz}\right).
\end{align}
\end{lems}
{We consider the following functionals
\begin{align*}
\bar{ \mathcal{J}}_{\eps,\lambda}(u)=
\begin{cases}
\mathcal{J}_{\eps,\lambda}(u) + \int_{\partial \O} \be{\gamma_0(u)} \d \mathcal{H}^{n-1} \quad & \text{for}~ u \in BV(\O)\cap L^2(\O),\\
+\infty  & \text{for}~ u \in BV(\O)\setminus L^2(\O),
\end{cases}
\end{align*}
and (for $\eps=0$)
\begin{align*}
\bar{ \mathcal{J}}_\lambda (u)=
\begin{cases}
\mathcal{J}_\lambda (u) + \int_{\partial \O} \be{\gamma_0(u)} \d \mathcal{H}^{n-1} \quad &\text{for}~ u \in BV(\O)\cap L^2(\O),\\
+\infty  &\text{for}~ u \in BV(\O)\setminus L^2(\O),
\end{cases}
\end{align*}
where $\gamma_0(u) $ is the trace of $u$ on the boundary and $\d \mathcal{H}^{n-1}$
is the Hausdorff measure.
The functionals $\bar{\mathcal{J}}_{\eps,\lambda}$ and $\bar{\mathcal{J}}_{\lambda}$ are  both convex and lower semicontinuous on $\L$ and the lower semicontinuous hulls of $\bar{\mathcal{J}}_{\eps,\lambda}\vert_{\Hz}$ and $\bar{\mathcal{J}}_{\lambda}\vert_{\Hz}$, respectively, cf. \cite[Proposition 11.3.2]{book_attouch}.}

As in \cite{our_paper} we interpret (\ref{TVF}), (\ref{eps.TVF}) as stochastic variational inequalities.
\begin{defs}\label{def_svi}
Let $0  < T < \infty$, {$\eps \in [0,1]$} and {$x_0 \in L^2(\Omega,\F_0;\L)$ and $g \in \L$}.
Then an $\F_t$-{adapted} stochastic process {$\Xee \in L^2(\Omega; C([0,T];\L))\cap  L^1(\Omega; L^1((0,T);BV(\O)))$ 
(denoted by $X \in L^2(\Omega; C([0,T];\L))\cap  L^1(\Omega; L^1((0,T);BV(\O)))$ for $\eps=0$)}
is called a {SVI  solution} of (\ref{eps.TVF}) (or (\ref{TVF}) if $\eps=0$) if $\Xee(0)=x_0$ ($X(0)=x_0$), and
for each $(\F_t)$-progressively measurable process $G\in L^2(\Omega \times (0,T),\L)  $ and for each $(\F_t)$-adapted $\L$-valued process $Z$
with $\P$-a.s. continuous sample paths, s.t, $Z \in L^2(\Omega \times (0,T);\Hz)$, which satisfy the equation 
\begin{align}\label{test}
\d Z(t)= -G(t) \d t +Z(t)\d W(t), ~ t\in[0,T],
\end{align}
it holds for {$\eps \in (0,1]$} that
\begin{align}\label{reg.SVI}
\frac{1}{2}& \E{\nos{\Xee(t)-Z(t)}}+\E{\intt {\bar{\mathcal{J}}_{\eps,\lambda}}(\Xee(s)) \d s} \nonumber\\
&\leq  \frac{1}{2} \E{\nos{x_0-Z(0)}}+\E{\intt { \bar{\mathcal{J}}_{\eps,\lambda}}(Z(s)) \d s}  \\
&+ \frac{1}{2}\E{\intt \nos{\Xee(s)-Z(s)} \d s}
+\frac{1}{2}\E{\intt \ska{\Xee(s)-Z(s),G} \d s}\,,\nonumber
\end{align}
and analogously for $\eps=0$ it holds that
\begin{align}\label{SVIeps0}
\frac{1}{2}& \E{\nos{X(t)-Z(t)}}+\E{\intt { \bar{\mathcal{J}}_{\lambda}}(X(s)) \d s} \nonumber\\
&\leq  \frac{1}{2} \E{\nos{x_0-Z(0)}}+\E{\intt { \bar{\mathcal{J}}_{\lambda}}(Z(s)) \d s}  \\
&+ \frac{1}{2}\E{\intt \nos{X(s)-Z(s)} \d s}
+\frac{1}{2}\E{\intt \ska{X(s)-Z(s),G} \d s}\nonumber.
\end{align}
\end{defs}
The next theorem shows that the solutions of the regularized problem \eqref{reg.TVF} 
converge to the SVI solution of \eqref{TVF} for $\eps, n \rightarrow \infty, \delta \rightarrow 0$;
the proof of the theorem follows as \cite[Theorem 3.2]{our_paper}. 
\begin{thms}\label{Thm.SVI}
Let $0  < T < \infty$ and {$x_0 \in L^2(\Omega,\F_0;\L)$,   $ g \in \L$} be fixed and consider $\Hz$-approximating sequences 
$\{x^n_0\}_{n\in \mathbb{N}} \subset L^2(\Omega,\F_0;\Hz)$, $\{g^n\}_{n\in \mathbb{N}} \subset \Hz$,
s.t.  $x^n_0 \rightarrow x_0$, $g^n \rightarrow g$ in $L^2(\Omega,\F_0; \L)$ for $n\rightarrow \infty$.
Let $\{\Xe_n\}_{\delta>0}$ be the  variational solutions of \eqref{reg.TVF} associated with  $x_0^n,g^n$, $\eps\in(0,1]$ and $\delta >0$.
Then  $\Xe_n$ converges to the unique SVI variational solution $X$ of (\ref{TVF})
in $L^2(\Omega;C([0,T];\L))$ for $\eps\rightarrow 0, n \rightarrow \infty,\delta \rightarrow 0 $, i.e.,
\begin{align}\label{epsilon goes to 0}
 \lim_{\eps \rightarrow 0}  \lim_{n \rightarrow \infty}  \lim_{\delta \rightarrow 0}\E{\sup_{t \in [0,T]}\nos{\Xe_n(t)-X(t)}}=0.
\end{align}
%Furthermore, the following estimate holds
%\begin{align}\label{stability_inequality}
%&\E{\nos{X_1(t)-X_2(t)}}\leq C\left(\E{\nos{x^1_0-x^2_0}}+\nos{g^1-g^2} \right)\quad \mathrm{for\,\,all\,\,} t \in [0,T]\,,
%\end{align}
%where $X_1$ and $X_2$ are SVI solutions of (\ref{TVF}) with $x_0\equiv x^1_0$, $g\equiv g^1$  and $x_0\equiv x^2_0$, $g\equiv g^2$, respectively.
\end{thms}

%%%%%%%%%%%%%%%%%%%%%%%%%%%%%%%%%%%%%%%%%
%%%%%%%%%%%%%%%%%%%%%%%%%%%%%%%%%%%%%%%%%%%%%%%%%%%%%%%%%%%%%%%%%%%%%%%%%%%%%%%%%%%%%%%%%%%%%%%%%%%%%%%%%%%%%%%%%%%%%%%%%%%#
%%%%%%%%%%%%%%%%%%%%%%%%%%%%%%%%%%%%%%%%%%%%%%%%%%%%%%%%%%%%%%%%%%%%%%%%%%%%%%%%%%%%%%%%%%%%%%%%%%%%%%%%%%%%%%%%%%%%%%%%%%%
\section{Semi-discretization in time}\label{sec_semi_num}
For $N \in \N$ we consider a partition of the time interval $t_i=i\tau$  for $i=0,\ldots,N$
with the time-step $\tau = T/N$, and denote the discrete Wiener increments as $\Delta_i W= W(t_i)-W(t_{i-1})$.

The implicit time-discrete approximation of (\ref{reg.TVF}) is defined as follows:
set $X_{\delta,n}^0=x_0^n$ and determine $\Xii\in \Hz$, $i=1,\dots, N$ as the solution of
\begin{align}\label{semi_eps_TVF}
\ska{\Xii, \Phi}&=\ska{\Xmi,\Phi}  -\tau \delta\ska{\nabla \Xii, \nabla \Phi} -\tau \ska{\fe{\Xii},\nabla\Phi }   \\ \nonumber
&\quad-\tau\lambda\ska{\Xii -g_n,\Phi}+\ska{\Xmi,\Phi}\Delta_i W ~~~\qquad \forall \Phi \in \Hz.
\end{align}
The existence, uniqueness and measurability of $\{\Xii\}_{i=1}^N$ can be shown via finite dimensional Galerkin approximation; we summarize the main steps below: 
\begin{itemize}
\item consider a finite dimensional subspace $\mathbb{V}_m$ and the corresponding Galerkin approximation $X_{\delta,n,m}^{i}\in\mathbb{V}_m$ of the solution $\Xii$ of \eqref{semi_eps_TVF};
\item proceed by induction: assuming that an $\mathcal{F}_{t_{i-1}}$-measurable solution $X_{\delta,n,m}^{i-1} \in \mathbb{V}_m$ exists,
the existence of an $\mathcal{F}_{t_{i}}$-measurable solution $X_{\delta,n,m}^{i}$ follows by Brouwer's fixed point theorem
and the uniqueness by the monotonicity property (\ref{eps.convexity.inequality}), cf. \cite[Lemma 4.3]{our_paper};
\item for any $m\in \mathbb{N}$ the Galerkin approximation $\{X_{\delta,n,m}^{i}\}_{i=1}^n$  satisfies the same a priori estimates as in Lemma~\ref{Lemma_Discrete a priori estimates} below;
\item by the (uniform in $m$) a priori estimates it holds that $X_{\delta,n,m}^{i}\rightharpoonup X_{\delta,n}^{i}$ for $m\rightarrow \infty$. Furthermore,
by the monotonicity (\ref{eps.convexity.inequality}) it follows that the limit $X_{\delta,n}^{i}$ is unique and satisfies (\ref{semi_eps_TVF}), cf., Lemma~\ref{lemma_Limiten_Gleichung} below.
\end{itemize}

In the next lemma we state the stability properties of the time-discrete solution of the scheme (\ref{semi_eps_TVF})
which are discrete analogues of estimates in Lemma~\ref{lem_eps_exist}~and Lemma~\ref{laplace_energy_estimate}.
Later on, we will consider sequences $\{x_0^n\}_{n\in\mathbb{N}}$, $\{g_n\}_{n\in\mathbb{N}}$
which are uniformly bounded in $\L$ but not in $\Hz$. Hence, in the following we suppress the dependence of the constants on the data in
(\ref{discrete_energy_estimate_viscTVF}) but not in (\ref{discrete_H1_estimate_viscTVF}).
\begin{lems}\label{Lemma_Discrete a priori estimates}
 Let $x_0^n \in L^2(\Omega,\F_0;\Hz)$ and $g_n\in\Hz$ be given.
Then there exists a constant $C \equiv C(\E{\|x_0^n\|_{\L}}, \|g_n\|_{\L}) > 0$ such that for any $\tau>0$
the solution of scheme (\ref{semi_eps_TVF}) satisfies
\begin{align}\label{discrete_energy_estimate_viscTVF}
\max_{i=1,\ldots,N}\E{\nos{\Xii}}+&\frac{1}{4}\E{\sum_{k=1}^N\nos{X_{\delta,n}^{k}-X_{\delta,n}^{k-1}}} 
\nonumber\\
 +&\tau \E{\sum_{k=1}^N \mathcal{J}_{\eps}(X_{\delta,n}^{k})} +\frac{\tau \lambda}{2}\E{\sum_{k=1}^N \nos{X_{\delta,n}^{k}}}\leq  C\,,
\end{align}
and a constant $C_{n} \equiv C( \mathbb{E}[\|x_0^n\|_{\Hz}], \|g_n\|_{\Hz}) > 0$ such that for any $\tau>0$
\begin{align}\label{discrete_H1_estimate_viscTVF}
\max_{i=1,\ldots,N}\E{\nos{\nabla \Xii}}+ \E{\sum_{k=1}^N\nos{\nabla(X_{\delta,n}^{k}-X_{\delta,n}^{k-1})}} +  \tau \delta \E{\sum_{k=1}^N\nos{\Delta X_{\delta,n}^{k}}} \leq  C_{n}.
\end{align}
\end{lems}
%%%%%%%%%%%%%%%%%%%%%%%%%%%%%%%%%%%%%%%%%%%%%%%%%%%%%%%%
\begin{proof} 
We set $\Phi =\Xii$ \eqref{semi_eps_TVF} and use the identity
$2(a-b)a =a^2 - b^2 + (a-b)^2$
to get for  $i=1,\ldots,N$
\begin{align}\label{num.visc.energie.estimate}
&\frac{1}{2}\nos{\Xii} +\frac{1}{2}\nos{\Xii-\Xmi}+\tau \delta \nos{\nabla \Xii}+\tau\ska{\fe{\Xii},\nabla \Xii} \nonumber\\
&=\frac{1}{2}\nos{\Xmi}-\tau \lambda\left(\nos{\Xii}-\ska{ g^{n},\Xii}\right) +\ska{\Xmi,\Xii}\Delta_i W. 
\end{align}
We take expectation in (\ref{num.visc.energie.estimate})
and use the properties of the Wiener increments $\E{\Delta_i W}=0$, $\E{ \bes{\Delta_i W}}=\tau$ and the independence of $\Delta_i W$ and $\Xmi$
to estimate the stochastic term as
\begin{align*}
\E{\ska{\Xmi,\Xii}\Delta_i W}&=\E{\ska{\Xmi,\Xii-\Xmi}\Delta_i W}+\E{\ska{\Xmi,\Xmi}\Delta_i W}
\\
\leq& \E{\frac{1}{4}\nos{\Xmi-\Xii}+\nos{\Xmi}\bes{\Delta_i W}} + \E{\|\Xmi\|^2}\E{\Delta_i W}
\\
=&\frac{1}{4}\E{\nos{\Xii-\Xmi}}+\tau\E{\nos{\Xmi}}.
\end{align*}
From (\ref{num.visc.energie.estimate}) by the convexity of $\mathcal{J}_\eps$ and using $\mathcal{J}_\eps(0)=\eps \be{\O}$ it follows that
\begin{align*}%\label{num.visc.energie.estimate2}
\frac{1}{2}\E{\nos{\Xii}} +&\frac{1}{4}\E{\nos{\Xii-\Xmi}}+\frac{\tau\lambda}{2}\E{\nos{\Xii}}+ \tau \delta \nos{\nabla \Xii} +\tau \E{\mathcal{J}_\eps(\Xii)}\\
\leq& \tau \eps \be{\O}+ \frac{1}{2}\E{\nos{\Xmi}}+\tau\E{\nos{\Xmi}}+\tau\lambda{\nos{g^{n}}}\,. \nonumber
\end{align*}
We sum up the above inequality for $k=1,\ldots,i$ and obtain
\begin{align}\label{num.visc.energie.estimate3}
\frac{1}{2}\E{\nos{\Xii}} +&\frac{1}{4}\E{\sum_{k=1}^i\nos{X_{\delta,n}^{k}-X_{\delta,n}^{k-1}}}+\frac{\tau \lambda}{2}\E{\sum_{k=1}^i \nos{ X_{\delta,n}^{k}}}\nonumber\\+&{\tau \delta\E{\sum_{k=1}^i \no{\nabla X_{\delta,n}^{k}}}}+\tau \E{\sum_{k=1}^i\mathcal{J}_\eps (X_{\delta,n}^{k})}
\\
\nonumber  
\leq& T \eps \be{\O}+\frac{1}{2}\E{\nos{x^n_0}}+T\lambda{\nos{  g^{n}}}+\tau\E{\sum_{k=1}^{i}\nos{X_{\delta,n}^{k-1}}} \,.
\end{align}
Then (\ref{discrete_energy_estimate_viscTVF}) follows from (\ref{num.visc.energie.estimate3})
after an application of the discrete Gronwall lemma.
%\begin{align*}
%\max_{i=1,\ldots,N}\E{\nos{ \Xii}} \leq C\Big(T \eps \be{\O}+\E{\nos{x_0}}+2T\lambda{\nos{g_n}}\Big).
%\end{align*}
%We substitute the above estimate into the right-hand side of \eqref{num.visc.energie.estimate3} 
%to conclude (\ref{discrete_energy_estimate_viscTVF}).
%\begin{align*}
%\max_{i=1,\ldots,N}\E{\nos{\Xii}}+2\tau \delta \E{\sum_{k=1}^i\nos{\nabla X_{\eps,\delta}^k}}+\tau \lambda\E{\sum_{k=1}^i \nos{ X_{\eps,\delta}^k}} \leq C.
%\end{align*}

To show the estimate (\ref{discrete_H1_estimate_viscTVF}) we proceed formally, the calculations can be made rigorous via finite dimensional Galerkin approximation, cf. \cite[Lemma~3.2]{our_paper}.
We set $\Phi = -\Delta \Xii$ in (\ref{semi_eps_TVF}), use integration by parts
and proceed analogously to the first part of the proof.

{As in the proof of \cite[Lemma~3.2]{our_paper} we deduce that}
\begin{align}\label{discrete_resolvent_estimate}
 \ska{\fe{\Xii},\nabla (- \Delta \Xii)} \geq 0.
\end{align}
Hence, we neglect the above term and conclude that
\begin{align*}
\frac{1}{2}& \E{\nos{\nabla \Xii}} +\frac{1}{4}\E{\sum_{k=1}^i\nos{\nabla(X_{\eps,n}^{k}-X_{\eps,n}^{k-1})}}+\frac{\tau \lambda}{2}\E{\sum_{k=1}^i \nos{ \nabla X_{\eps,n}^{k}}}
\\
&+ \tau \delta\E{\sum_{k=1}^i \no{\Delta X_{\delta,n}^{k}}} \leq  \frac{1}{2}\E{\nos{\nabla x^n_0}}+T\lambda{\nos{\nabla   g_n}}+\tau\E{\sum_{k=1}^{i}\nos{\nabla X_{\eps,n}^{k-1}}}\,.
\end{align*}
Estimate (\ref{discrete_H1_estimate_viscTVF}) then follows after an application of the discrete Gronwall lemma.
\end{proof}

\begin{bems}
The proof of the convergence of the numerical approximation given in \cite{our_paper} relies on the stronger a priori estimate
\eqref{discrete_H1_estimate_viscTVF}.
The above proof of the estimate \eqref{discrete_H1_estimate_viscTVF} requires property \eqref{discrete_resolvent_estimate} to hold.
So far, the proof of the spatially discrete counterpart of the estimate \eqref{discrete_resolvent_estimate} is restricted to 
spatial dimension $d = 1$ \cite[Lemma 3.1]{stvf_erratum}.
In the proof of the convergence of the fully discrete numerical approximation below 
we circumvent the lack of a (rigorous) discrete counterpart of \eqref{discrete_resolvent_estimate} for $d>1$ by considering the time-discrete problem (\ref{semi_eps_TVF}) as an intermediate step.
\end{bems}
%%%%%%%%%%%%%%%%%%%%%%%%%%%%%%%%%%%%%%%%%%%%%

We define piecewise constant time-interpolants of the numerical solution
$\{\Xii\}_{i=0}^{N}$ of (\ref{semi_eps_TVF}) for $t\in [0,T]$ as
\begin{align}\label{eps_delta_interpol1}
\Yc(t)= \Xii \quad  \mathrm{if}\quad t\in (t_{i-1},t_i]
\end{align}
and
\begin{align}\label{eps_delta_interpol2}
\Ycm(t)= \Xmi\quad \mathrm{if}\quad t\in [t_{i-1},t_i)\,.
\end{align}
We note that (\ref{semi_eps_TVF}) can be reformulated as
\begin{align}\label{Integralformulation}
 &\ska{\Yc(t),\Phi}+\left\langle\int_0^{\theta_{+}(t)} A^{\delta}\Yc(s) \d s,\Phi\right\rangle \nonumber \\
 &=\ska{X_{\eps,n}^0,\Phi}+\ska{\int_0^{\theta_+(t)} \Ycm(s) \d W(s),\Phi} \qquad \mathrm{for}\,\, t\in [0,T],\,\, \Phi \in \Hz,
\end{align}
where $\theta_+(0)=0$ and $\theta_+(t)=t_i$ if  $t\in (t_{i-1},t_{i}]$.

Estimates \eqref{discrete_energy_estimate_viscTVF}, \eqref{discrete_H1_estimate_viscTVF} imply the bounds
\begin{align}\label{Constant interpolation estimate}
\sup_{t\in [0,T]}\E{\nos{\Yc(t)}} &\leq C, &\sup_{t\in [0,T]}\E{\nos{\Ycm(t)}} \leq C,\\
~  \delta \E{\int_0^T \nos{\nabla\Yc(s)}\d s} &\leq C.\nonumber
\end{align}
Furthermore, \eqref{Constant interpolation estimate} and \eqref{a_bnd} imply
\begin{align}\label{Aed Abschaetzung}
\E{\int_0^T \nos{A^{\delta} \Yc(s)}_{\Hm} \d s} \leq C.
\end{align}
The estimates in (\ref{Constant interpolation estimate}) for fixed  $n \in \N$, $\eps,\delta >0$
imply the existence of a subsequence, still denoted by $\{\Yc\}_{\tau>0}$,
and a $Y \in L^2(\Omega\times (0,T);\L)\cap L^2(\Omega\times (0,T);\Hz)\cap L^{\infty}((0,T);L^2(\Omega;\L)$, s.t., for $\tau \rightarrow 0$
\begin{align}\label{limit_process}
\Yc &\weak Y ~\text{in}~ L^2(\Omega\times (0,T);\L), \nonumber\\
\Yc &\weak Y ~\text{in}~ L^2(\Omega\times (0,T);\Hz),\\
\Yc &\weak^* Y ~\text{in}~ L^{\infty}((0,T);L^2(\Omega;\L)) \nonumber.
\end{align}
In addition, there exists $\nu \in L^2(\Omega;\L)$ such that $\Yc(T) \rightharpoonup \nu$ 
in $L^2(\Omega;\L)$ as $\tau \rightarrow 0$ and estimate (\ref{Aed Abschaetzung}) 
implies the existence of $a^{\delta} \in L^2(\Omega\times (0,T);\Hm)$, s.t.,
\begin{align}\label{lim_a}
A^{\delta} \Yc &\weak a^{\delta} ~\text{in}~ L^2(\Omega\times (0,T);\Hm)\quad \mathrm{for} \,\,\tau\rightarrow 0.
\end{align} 
Furthermore, the estimates in (\ref{Constant interpolation estimate}) for fixed  $n \in \N$, $\eps,\delta > 0$ imply
the existence of a subsequence, still denoted by $\{\Ycm\}_{\tau>0}$,
and of $Y^- \in L^2(\Omega\times (0,T);\L)$, s.t.,
\begin{align*}%\label{limit_process_minus}
\Ycm &\weak Y^- ~\text{in}~ L^2(\Omega\times (0,T);\L)\quad \mathrm{for} \,\,\tau\rightarrow 0.
\end{align*}
Finally, inequality \eqref{num.visc.energie.estimate3} implies
\begin{align*}%\label{same_weak}
\lim_{\tau \rightarrow 0}\E{\int_0^T \nos{\Yc(s)-\Ycm(s)} \d s }=&\lim_{\tau \rightarrow 0}\tau \E{\sum_{k=1}^N\nos{X_{\delta,n}^k-X_{\delta,n}^{k-1}}} \nonumber\\
\leq& \lim_{\tau \rightarrow 0} C\tau =0\,.  
\end{align*}
which shows that the weak limits of $Y$ and $Y^-$ coincide. % in $L^2(\Omega;L^2((0,T);\L))$.

From the above convergence properties we deduce by standard arguments, cf. \cite[Lemma 4.6]{our_paper},
that the solutions of the semi-discrete scheme (\ref{semi_eps_TVF}) converge to the unique variational solution of (\ref{reg.TVF}) for $\tau\rightarrow 0$.
\begin{lems}\label{lemma_Limiten_Gleichung}
Let $x_0^n \in L^2(\Omega,\F_0;\Hz)$ and $g_n\in\Hz$ be given, let $\eps,\delta,\lambda>0$, {$n \in \N$} be fixed. 
Further, let $\Xdn$ be the unique variational solution of \eqref{reg.TVF}
and $\Yc$, $\Ycm$ be the respective time-interpolant (\ref{eps_delta_interpol1}), (\ref{eps_delta_interpol2}) of the numerical solution $\{\Xii\}_{i=1}^N$ of \eqref{semi_eps_TVF}.
Then $\Yc$, $\Ycm$ converge to $\Xdn$ for $\tau \rightarrow 0$ in the sense that
the weak limits from (\ref{limit_process}), (\ref{lim_a}) satisfy $Y\equiv \Xdn$, $a^{\delta}\equiv A^{\delta} Y \equiv A^{\delta} \Xdn$ and $\nu=Y(T)\equiv \Xdn(T)$.
In addition, it holds for almost all $(\omega,t) \in \Omega\times (0,T)$ that
\begin{align*}%\label{Limiting Gleichung}
Y(t)=Y(0){-}\intt A^{\delta} Y(s) \d s+\intt Y(s)\d W(s),
\end{align*}
and there is an $\L$-valued continuous modification of $Y$ (denoted again by $Y$) such that for all $t \in [0,T]$ 
\begin{align}\label{Ito-Formule_fuer_Limiten}
\frac{1}{2}\nos{Y(t)}= & \frac{1}{2}\nos{Y(0)}{-}\intt \langle A^{\delta} Y(s),Y(s) \rangle +\frac{1}{2}\nos{Y(s)} \d s
\\ \nonumber 
& +\intt (Y(s),Y(s))\d W(s).
\end{align}
\end{lems}
%%%%%%%%%%%%%%%%%%%%%%%%%%%%%%%%%%%%%%%%%%%%%%%%%%%%%%%%%%%%%%%%%%%%%%%%%%%%%%%%%%%%%%%%%%%%%%%%%%%%%%%%%%%%%%%%%%%%%%%%%%%%%%%%%%%%
{The strong monotonicity property \eqref{Monotonicity} of the operator $A^{\delta}$ implies strong convergence of the time-discrete
approximation in  $L^2(\Omega\times(0,T);\L)$, cf. \cite[Lemma 4.7]{our_paper}.}
\begin{lems}\label{Lemma_Convergence_num.vis.Scheme}
Let $x_0^n \in L^2(\Omega,\F_0;\Hz)$ and $g_n\in\Hz$ be given,
let $\eps,\delta, \lambda>0$, {$n \in \N$} be fixed. Furthermore, let $\Xdn$ be the variational solution of \eqref{reg.TVF} 
and $\Yc$ be the time-interpolants (\ref{eps_delta_interpol1}) of the time-discrete solution $\{\Xii\}_{i=1}^N$ of \eqref{semi_eps_TVF}.
Then
\begin{align}
\lim_{\tau \rightarrow 0}\nos{\Xe_n-\Yc}_{L^2(\Omega \times (0,T);\L)}\rightarrow 0.
\end{align}
\end{lems}

\section{Full Discretization}\label{sec_full_num}
%{\blue We consider a fully disrete scheme which employs a standard $\mathbb{H}^1_0$-conforming finite element method over a quasi-uniform triangulation $\mathcal{T}_h$,
%see, e.g., \cite{BrennerS02}, \cite{Prohl_TVF_numerics}.}
Given a quasi-uniform triangulation $\mathcal{T}_h$ of $\O$ we consider the $\mathbb{H}^1_0$-conforming finite element space of globally continuous piecewise linear functions over $\mathcal{T}_h$ given as
\begin{align*}
\mathbb{V}_h=\left\{w_h \in C^0(\O) : w_h \vert_{T} \in \mathcal{P}_1(T) ~\forall T \in \mathcal{T}_h\right\} \subset \Hz.
\end{align*}
The orthogonal $\L$-projection $\Pi_h : \mathbb{H}^1 \rightarrow \mathbb{V}_h$ is defined as
$$
(v-\Pi_h v, w_h) = 0\qquad \forall w_h \in \mathbb{V}_h\,.
$$
It is well-known, see e.g., \cite{BrennerS02}, \cite{Prohl_TVF_numerics}, that the projection operator satisfies the following interpolation and stability properties for $\psi \in \mathbb{H}^1$:
% see \cite[Theorem 4.6]{Bartels_book}:
\begin{align}\label{Interpolation_properties_1}
\no{\psi-\Pi_h \psi}\leq Ch\no{\nabla \psi}~~~\text{and}~~~\|\Pi_h \psi\|_{\mathbb{H}^1}\leq C \|\psi\|_{\mathbb{H}^1}\,. 
\end{align}
For $\psi \in \mathbb{H}^2$ one has the following estimate%, \cite[Remark 4.12]{Bartels_book}
\begin{align}\label{Interpolation_properties_2}
\no{\psi-\Pi_h \psi}+h\no{\nabla [\psi-\Pi_h \psi]}\leq Ch^2\no{\nabla^2 \psi}. 
\end{align}

By the estimate \eqref{discrete_H1_estimate_viscTVF} we deduce from \eqref{Interpolation_properties_2} that
\begin{align}\label{Interpolation_convergence}
\tau \sum_{i=1}^N\E{\nos{\nabla\big(X_{\delta,n}^{i}-\Pi_h X_{\delta,n}^{i}\big)}} \leq C_n\delta^{-1}h^{2}\,,
\end{align}
uniformly for all $\tau>0$. %, i.e., the constant $C_n$ is independent of $\tau$.

Given $\mathbb{H}^1$-regular data $x^n_0$, $g_{n}$ we consider the following auxiliary fully discrete numerical scheme.
Set $X^0_{\eps.n,h}=\Pi_h x^n_0$, $g_{n,h}=\Pi_h g_n$, $\tau=T/N$ and
determine $\Xi \in \mathbb{V}_h$, $i=1,\dots, N$  as the solution of
\begin{align}\label{full_eps_TVF}
\ska{\Xi,\vh}&=\ska{\Xmin,\vh}-\tau \ska{\fe{\Xi},\nabla \vh } \\
&-\tau\lambda\ska{\Xi -g_{n,h},\vh}+\ska{\Xmin,\vh}\Delta_i W &&\forall \vh \in \mathbb{V}_h \ \nonumber.
\end{align}

The existence, uniqueness and measurability of the numerical solution $\{\Xi\}_{i=1}^N$ follows as in \cite[Lemma 5.3]{our_paper}.

In the next lemma we state the stability properties of the auxiliary numerical scheme (\ref{full_eps_TVF}).
The proof of the estimate is a direct counterpart of the proof of \eqref{discrete_energy_estimate_viscTVF} and is therefore omitted. 
\begin{lems}\label{eps.num.energy.estimates}
Let $x^n_0, g_n\in\Hz$ and $T> 0$.
Then there exists a constant $C\equiv C(T)$ such that 
the solutions of scheme \eqref{full_eps_TVF} satisfy for any $\eps,h\in (0,1]$, $N\in \mathbb{N}$ 
\begin{align}\label{disc_ener}
\max_{i=1,\ldots,N}\E{\nos{\Xi}}+&\frac{1}{4}\E{\sum_{k=1}^N\nos{X_{\eps,n,h}^{k}-X_{\eps,n,h}^{k-1}}} 
\nonumber\\
 +&\tau \E{\sum_{k=1}^N \mathcal{J}_{\eps}(X^{k}_{\eps,n,h})} +\frac{\tau \lambda}{2}\E{\sum_{k=1}^N \nos{X_{\eps,n,h}^{k}}}\leq  C\,.
\end{align}
\end{lems}

The next lemma provides an estimate for the difference between 
the solutions of the auxiliary fully discrete numerical scheme \eqref{full_eps_TVF} and the solutions of its semi-discrete counterpart \eqref{semi_eps_TVF}.
\begin{lems}\label{Differenece_time_space}
Let $\eps > 0$, $\delta>0$, $n\in\mathbb{N}$ be fixed. Let $\Xii$ be the solution of the semi-discrete scheme \eqref{semi_eps_TVF} and let $\Xi$ be the numerical solution of the fully-discrete scheme \eqref{full_eps_TVF}. 
Then the following estimate holds for $0<\tau \leq \frac{1}{2}$:
\begin{align*}
\max_{i=1,\ldots,N} \E{ \nos{\Xii-\Xi}} 
& \leq C\left(C_n h+  C_n^{1/2}\delta^{-\frac{1}{2}}h+ C_n \delta + \lambda\nos{g_n-g_{n,h}}\right).
\end{align*}
\end{lems}
\begin{proof}
We set $Z^i=\Xii-\Xi$ and observe the following equality 
\begin{align}\label{id0}
\ska{Z^i-Z^{i-1},\Pi_h Z^i}=\ska{\Pi_h(Z^i-Z^{i-1}),\Pi_h Z^i }
=\frac{1}{2}\nos{\Pi_h Z^i}+\frac{1}{2}\nos{\Pi_h (Z^i-Z^{i-1})}-\frac{1}{2}\nos{\Pi_h Z^{i-1}}\,,
\end{align}
where we used the elementary property of the orthogonal projection that $\ska{v,\Pi_h v} = \ska{\Pi_h v, \Pi_h v}$.

We set $\Phi = \vh= \Pi_h(\Xii-\Xi)$ in \eqref{full_eps_TVF}, \eqref{semi_eps_TVF} (note $\Pi_h\Xi =\Xi$) and obtain after subtracting the respective equations and using (\ref{id0})
\begin{align}\label{l2_ineq}
\frac{1}{2}\nos{\Pi_h Z^i}&+\frac{1}{2}\nos{\Pi_h (Z^i-Z^{i-1})}-\frac{1}{2}\nos{\Pi_h Z^{i-1}}
\nonumber \\
&+\tau\ska{\fe{\Xii}-\fe{\Xi},\nabla (\Xii-\Xi)}
\nonumber\\
&-\tau\ska{\fe{\Xii}-\fe{\Xi},\nabla (\Xii-\Pi_h\Xii)}
 \\
&+\tau \lambda\left( \ska{\Xii-g_n,\Pi_h Z^i}-\ska{\Xi-g_{n,h},\Pi_h Z^i}\right)
\nonumber \\
\nonumber
=& \ska{Z^{i-1},\Pi_h Z^i }\Delta_iW  + \tau \delta\ska{\Delta \Xii, Z^i }. 
\end{align}
By \eqref{eps.convexity.inequality} the fourth term on the left hand side is positive and can be neglected.
We estimate the fifth term on the left hand side in \eqref{l2_ineq} using $\be{\fe{\cdot}} \leq 1$ and the Cauchy-Schwarz inequality as
 \begin{align}\label{non_linear_ineq}
 & \ska{\fe{\Xii}-\fe{\Xi},\nabla (\Xii-\Pi_h\Xii)} 
\\ \nonumber 
 & \qquad \leq  2\be{\O}^{\frac{1}{2}}\no{\nabla (\Xii-\Pi_h\Xii)}
\,.
\end{align}  
Using the Cauchy-Schwarz and Young inequalities the last term on the left-hand in (\ref{l2_ineq}) can be estimated as
\begin{align*}
\ska{\Xii-g_n,\Pi_h Z^i} -\ska{\Xi-g_{n,h},\Pi_h Z^i}
&=\ska{g_{n,h}-g_n,\Pi_h Z^i} +\ska{\Xii- \Xi,\Pi_h Z^i}
\\
&\geq 
\frac{1}{2}\nos{\Pi_h Z^i}-\frac{1}{2}\nos{g_n-g_{n,h}}\,,
\end{align*}
and the last term on the right-hand side as
$$
 \delta\ska{-\Delta X_{\delta,n}^{i}, \Pi_h Z^i}
\leq \frac{\delta^2}{2} \|\Delta X_{\delta,n}^{i}\|^2 + \frac{1}{2}\nos{\Pi_h Z^i} \,.
$$
After substituting the above inequalities into \eqref{l2_ineq} we obtain
\begin{align}\label{l2_ineq_2}
\frac{1}{2}\nos{\Pi_h Z^i}+&\frac{1}{2}\nos{\Pi_h (Z^i-Z^{i-1})}-\frac{1}{2}\nos{\Pi_h Z^{i-1}}
\nonumber \\
\leq&  \ska{Z^{i-1} ,\Pi_h Z^i}\Delta_i W + \tau \be{\O}^{\frac{1}{2}}\no{\nabla (\Xii-\Pi_h\Xii)}
 \\ \nonumber
&+\frac{\lambda\tau}{2}\nos{g_n-g_{n,h}}
% \\ \nonumber&\quad 
+  \frac{\tau \delta^2}{2} \|\Delta X_{\delta,n}^{i}\|^2 + \frac{\tau}{2}\nos{\Pi_h Z^i}.
\end{align}
We estimate the stochastic term as
\begin{align*}
\E{\ska{Z^{i-1}, \Pi_h Z^i}\Delta_i W}&=\E{\ska{Z^{i-1},\Pi_h Z^{i}-\Pi_h Z^{i-1}}\Delta_i W+\ska{Z^{i-1},\Pi_h Z^{i-1}}\Delta_i W}
\\
&=\E{ \ska{\Pi_h Z^{i-1},\Pi_h Z^{i}-\Pi_h Z^{i-1}}\Delta_i W+\ska{\Pi_h Z^{i-1},\Pi_h Z^{i-1}}\Delta_i W}
\\
& \leq\E{ \frac{1}{2}\nos{\Pi_h Z^{i-1}}\bes{\Delta_i W} + \frac{1}{2}\nos{\Pi_h (Z^i-Z^{i-1})}}
\\
& = \frac{\tau}{2}\E{ \nos{\Pi_h Z^{i-1}}} + \frac{1}{2}\E{\nos{\Pi_h (Z^i-Z^{i-1})}}.
\end{align*}
Hence, we obtain after taking expectation in \eqref{l2_ineq_2} and summing over $i$ that
\begin{align}\label{estimate_1}
\frac{1}{2}\E{\nos{\Pi_h Z^i}} &\leq \frac{1}{2}\E{\nos{\Pi_h Z^{0}}}
+\frac{\tau}{2}\E{\sum_{k=1}^i \no{\Pi_h Z^{k-1}}^2}
\nonumber\\
 &\quad + 
2\tau \be{\O}^{\frac{1}{2}}  \E{\sum_{k=1}^i \no{\nabla[X_{\delta,n}^{k}-\Pi_h X_{\delta,n}^{k}]} }+
\frac{\tau}{2}\E{\sum_{k=1}^i \nos{\Pi_h Z^k}}
\\
&\quad +\frac{T\lambda}{2}\nos{g_n-g_{n,h}}  + \frac{\tau \delta^2}{2}  \E{\sum_{k=1}^i \|\Delta X_{\delta,n}^{i}\|^2}
\nonumber
\\
\nonumber
&= \mathrm{I+II+III+IV+V+VI}.
\end{align}
By the Cauchy-Schwartz inequality and \eqref{Interpolation_convergence} we obtain
\begin{align}\label{eq_deltaint}
\mathrm{III}
&\leq (T\be{\O})^{\frac{1}{2}}\left(\sum_{k=0}^i \tau\E{\nos{\nabla[X_{\delta,n}^{k}-\Pi_h X_{\delta,n}^{k}]}}\right)^{\frac{1}{2}}
\leq C \left(C_n \delta^{-1}h^2\right)^{\frac{1}{2}}\,.
%\nonumber \\
%& \leq C \max_{k=1,\ldots,N}\left(\E{\nos{\nabla[X_{\delta,n}^{k}-\Pi_h X_{\delta,n}^{k}]}}\right)^{\frac{1}{2}}
\end{align}
Estimate \eqref{discrete_H1_estimate_viscTVF} implies
$$
%{\red \frac{\delta^2}{2} \tau \E{\sum_{k=1}^i \|\Delta X_{\delta,n}^{i}\|^2} } 
\mathrm{VI} \leq C_n\delta,
$$
and since $X^0_{\eps,n,h}=\Pi_h x^n_0$ and $X^0_{\delta,n}= x^n_0$, we deduce
$$
I = \frac{1}{2}\| \Pi_h X^0_{\delta,n} - X^0_{\eps,n,h} \|^2 = 0\,.
$$
After substituting the above estimates for $I$, $III$, $VI$ into \eqref{estimate_1}, we obtain by the discrete Gronwall lemma for sufficiently small $\tau$ (e.g. $\tau \leq \frac{1}{2}$) that
 \begin{align*}
\max_{i=1,\dots,N}\E{\nos{\Pi_h Z^i}}
&\leq 
  C\left(\lambda \no{g_n-g_{n,h}}
%\\&\quad 
+  C_n^{1/2} \delta^{-\frac{1}{2}}h+  C_n\delta\right).
\end{align*}
The statement of the lemma then follows from the above estimate by (\ref{Interpolation_properties_1}) and (\ref{discrete_H1_estimate_viscTVF}), since
\begin{align*}
\E{\nos{Z^i}} &\leq 2\E{\nos{Z^i-\Pi_h Z^i}}+2\E{\nos{\Pi_h Z^i}} 
\\
& =2\E{\nos{\Xii-\Pi_h \Xii}}+2\E{\nos{\Pi_h Z^i}}
\\
&\leq C_n h+2\E{\nos{\Pi_h Z^i}}.
\end{align*}
\end{proof}

The fully discrete numerical approximation of (\ref{eps.TVF}) is constructed as follows.
For $x_0,g \in \mathbb{L}^2$ we set $X^{0}_{\eps,h}=\Pi_h x_0$ and $g_{h}=\Pi_h g$ and determine $\xii$, $i=1,\dots,N$ as the solution of:
\begin{align}\label{full_eps_TVF_L2_data}
\ska{\xii,\vh}=&\ska{\xmi,\vh}-\tau \ska{\fe{\xii},\nabla \vh } \\
&-\tau\lambda\ska{\xii -g_{h},\vh}+\ska{\xmi,\vh}\Delta_i W &&\forall \vh \in \mathbb{V}_h \ \nonumber.
\end{align}
The existence, uniqueness and measurability properties of the solutions of (\ref{full_eps_TVF_L2_data}) follow analogously as for the solutions of (\ref{full_eps_TVF}).

{In the next lemma we estimate the difference between the solutions of the fully discrete numerical scheme (\ref{full_eps_TVF_L2_data}) and the auxiliary scheme (\ref{full_eps_TVF}). }
%%%%%%%%%%%%%%%%%%%%%%%%%%%%%%%%%%%%%
\begin{lems}\label{Lemma_Difference-num.Schemes}
Let $x_0 \in L^2(\Omega,\F_0;\L)$ and $g\in\L$ be given.
Then for each $n \in \mathbb{N}$ there exists a constant $C\equiv C(T)>0$, 
such that for any $N\in\N$, $n\in \mathbb{N}$, $h,\eps \in (0,1]$ the following estimate holds for the difference of the numerical solutions of (\ref{full_eps_TVF}) and (\ref{full_eps_TVF_L2_data}):
\begin{align*}
\max_{i=1,\ldots,N}\E{\nos{\Xi-\xii}} \leq C\left(\E{\nos{x_0-x_0^{n}}}+ \lambda\nos{g-g_{n}}\right). 
\end{align*} 
\end{lems}
%%%%%%%%%%%%%%%%%%%%%%%%%%%%%%%%%%%%%
\begin{proof}
We  define $\Zii=\Xi -\xii$. After subtracting \eqref{full_eps_TVF} and \eqref{full_eps_TVF_L2_data} we get
\begin{align*}
\ska{\Zii,\vh}=& \ska{\Zmi,\vh}
%\\ &
-\tau \ska{\fe{\Xi}-\fe{\xii},\nabla \vh}
\\
&-\tau\lambda\ska{\Zii,\vh}-{\tau\lambda\ska{g_{h} - g_{n,h},\vh}}
%\\&
+\ska{\Zmi,\vh}\Delta_i W.
\end{align*}
We set $\vh=\Zi$ and obtain
\begin{align}\label{eqz}
\ska{\Zii-\Zmi,\Zii} 
= & -\tau \ska{\fe{\Xi}- \fe{\xii},\nabla\Zii }
\\\nonumber
&-\tau\lambda\nos{\Zii}-\tau\lambda\ska{g_{h}- g_{n,h},\Zii}
%\\&
+\ska{\Zmi,\Zii}\Delta_i W.
\end{align}
We rewrite the left-hand side in (\ref{eqz}) as
\begin{align*}
\ska{\Zii-\Zmi,\Zmi}= \frac{1}{2}\nos{\Zii}-\frac{1}{2}\nos{\Zmi}+\frac{1}{2}\nos{\Zii -\Zmi}\,,
\end{align*}
and by the Cauchy-Schwarz and Young inequalities we estimate
\begin{align*}
  &\tau \lambda \ska{g_{h}-g_{n,h} ,\Zii}\leq  \frac{\tau \lambda}{2} \nos{g_{h}-g_{n,h}} +\frac{\tau \lambda}{2}\nos{\Zii}.
\end{align*}
Furthermore, the convexity (\ref{eps.convexity.inequality}) implies that
\begin{align*}
-\tau \ska{\fe{\Xi}-\fe{\xii},\nabla(\Xi-\xii)}\leq 0.
\end{align*}
Using the above estimates we deduce from (\ref{eqz}) that
\begin{align}\label{zest1}
\frac{1}{2}&\nos{\Zii}+\frac{1}{2}\nos{\Zii -\Zmi} - \frac{1}{2}\nos{\Zmi}
+ \frac{\tau \lambda}{2}\nos{\Zmi}
\\
&\leq \frac{\tau \lambda}{2} \nos{g_{h}-g_{n,h}} +\ska{\Zmi,\Zii}\Delta_i W\,. \nonumber
\end{align}
We estimate the last term on the right-hand side above as
\begin{align*}
\ska{\Zmi,\Zii}\Delta_i W=\ska{\Zmi,\Zii-\Zmi}\Delta_i W+\nos{\Zmi}\Delta_i W\\
\leq \frac{1}{2}\nos{\Zii-\Zmi}+\frac{1}{2}\nos{\Zmi}\bes{\Delta_i W} +\nos{\Zmi}\Delta_i W.
\end{align*}
{We substitute the above identity into (\ref{zest1}), neglecting the positive term multiplied by $\lambda$ on the left-hand side and arrive at}
\begin{align*}
\frac{1}{2}\nos{\Zii} - \frac{1}{2}\nos{\Zmi} 
\leq &  \frac{\tau \lambda}{2} \nos{g_{h}-g_{n,h}} 
%\\&
+\frac{1}{2}\nos{\Zmi}\bes{\Delta_i W}+\nos{\Zmi}\Delta_i W.
\end{align*}
%We drop the positive term $\frac{\tau \lambda}{2} \nos{\Zi}$,
Hence, we sum the above inequality over $i$, take expectation and obtain 
\begin{align*}
\frac{1}{2}\E{\nos{\Zii}} 
%\leq &\frac{1}{2}\nos{Z^0_{\eps}}+\frac{\tau}{2}\sum_{k=1}^{i}\E{\nos{Z^k_{\eps}}}\\
\leq&\frac{1}{2}\E{\nos{Z^0_{\eps}}} +\frac{\tau}{2}\sum_{k=0}^{i-1}\E{\nos{Z^k_{\eps,h}}}+ \frac{T \lambda}{2} \nos{g_{h}-g_{n,h}}.
\end{align*}
Finally, an application of the discrete Gronwall lemma yields that
\begin{align*}%\label{m,n gronwall}
\max_{i=1,\ldots,N}\E{\nos{\Zii}} \leq C\left(\E{\nos{\Pi_h(x_0- x_0^{n})}}+ \lambda\nos{\Pi_h(g-g_{n})} \right),
\end{align*}
and the statement of the lemma follows by the stability of the $\L$-projection (\ref{Interpolation_properties_1}).
\end{proof}

We define piecewise constant time-interpolants
of the discrete  solutions $\{\Xii \}_{i=0}^N $ of \eqref{semi_eps_TVF}, $\{\Xi \}_{i=0}^N$ of (\ref{full_eps_TVF}) and $\{\xii \}_{i=0}^N$ of (\ref{full_eps_TVF_L2_data}) for $t\in[0,T)$ as
\begin{align}\label{eps_interpol}
\overline{X}_{\tau}^{\delta,n}=\Xii,~~ \overline{X}_{\tau,h}^{\eps,n}=\Xi, ~~\overline{X}_{\tau,h}^{\eps} = \xii\quad \mathrm{if}\quad t \in (t_{i-1},t_i]. 
\end{align}

In the next theorem we conclude the paper by showing the convergence of the fully discrete numerical approximation (\ref{full_eps_TVF_L2_data}) to the unique SVI solution of the total variation flow (\ref{TVF})
(cf. Definition~\ref{def_svi}).
%%%%%%%%%%%%%%%%%%%%%%%%%%%%%%%%%%%%%%%%%%%%%%%%%%%%%%%%%%%%%%%%%%%
%%%%%%%%%%%%%%%%%%%%%%%%%%%%%%%%%%%%%%%%%%%%%%%%%%%%%%%%%%%%%%%%%%%
%%%%%%%%%%%%%%%%%%%%%%%%%%%%%%%%%%%%%%%%%%%%%%%%%%%%%%%%%%%%%%%%%%%
\begin{thms}\label{Thm_Convergence_num.reg.Scheme}
Let $X$ be the SVI solution of \eqref{TVF} and let $\overline{X}_{\tau,h}^{\eps} $ be the time-interpolant (\ref{eps_interpol})
of the solutions of the fully-discrete scheme \eqref{full_eps_TVF_L2_data}.
Then %the following convergence holds true
\begin{align}\label{numconv}
\lim_{\eps\rightarrow 0}\lim_{\tau,h \rightarrow 0}\nos{X-\overline{X}_{\tau,h}^{\eps} }_{L^2(\Omega \times (0,T);\L)}\rightarrow 0.
\end{align}
 %under the condition that the  mesh size satisfies $h = C^* \tau^{\frac{1}{2}+\kappa}$ for an arbitrary $\kappa > 0$, $C^* > 0$.
\end{thms}
\begin{proof}
%Let $X_m$ be a varational solution for \eqref{TVF} with datum functions $x_0^m,g^m$ and
For $x_0 \in L^2(\Omega,\F_0;\L)$ and $g\in\L$  for $n\in \mathbb{N}$
we set $x_0^n = \mathcal{P}_n x_0$, $g_n= \mathcal{P}_n g$
where $\mathcal{P}_n :\L \rightarrow \mathbb{V}_n$ is the orthogonal $\L$-projection onto the finite dimensional eigenspace $\mathbb{V}_n = \text{span}\{e_0,\ldots, e_n\}\subset \Hz$.
By construction the sequences $\{x_0^n\}_{n\in \mathbb{N}}\subset \mathbb{H}^1_0$, $\{g_n\}_{n\in \mathbb{N}}\subset \mathbb{H}^1_0$ 
satisfy $x_0^n\rightarrow x_0\in L^2(\Omega,\F_0;\L)$, $n\in\mathbb{N}$, $g_n\rightarrow g\in\L$. 
Below, we consider (\ref{reg.TVF}), (\ref{full_eps_TVF}) with the data $x_0^n$, $g_n$ defined above.

%Below, we denote by $\Xe$ the solution of \eqref{reg.TVF} with the data $x_0^n$, $g_n$ replaced by $x_0$, $g$, respectively;
%the existence of $\Xe$ follows by standard arguments for monotone SPDEs, cf., for instance \cite{our_paper}.

By the triangle inequality we get
{\begin{align}\label{err_ineq}
\nonumber
\frac{1}{4}\nos{X-\overline{X}_{\tau,h}^{\eps} }_{L^2(\Omega \times (0,T);\L)} \leq& \nos{X- \Xe_n }_{L^2(\Omega \times (0,T);\L)} +\nos{\Xdn-\Yc}_{L^2(\Omega \times (0,T);\L)}
\nonumber\\
&+\nos{\Yc-\Xc}_{L^2(\Omega \times (0,T);\L)}
+\nos{\Xc-\overline{X}_{\tau,h}^{\eps}}_{L^2(\Omega \times (0,T);\L)}
\\
\nonumber
& =:\mathrm{I+II+III+IV}.
\end{align}
}
From Theorem  \ref{Thm.SVI} it follows that
\begin{align*}
 \lim_{\eps \rightarrow 0} \lim_{n \rightarrow \infty}\lim_{\delta \rightarrow 0} \, \mathrm{I} =  \lim_{\eps \rightarrow 0}   \lim_{n \rightarrow \infty}\lim_{\delta \rightarrow 0} \E{\nos{X-\Xe_n}}=0.
\end{align*}
By Lemma~\ref{Lemma_Convergence_num.vis.Scheme} we deduce for the second term that
\begin{align*}
\lim_{\tau\rightarrow 0}\mathrm{II}=\lim_{\tau\rightarrow 0}\E{\nos{\Xdn-\Yc}} =0.
\end{align*} 
{For the third term we get by Lemma \ref{Differenece_time_space}} and (\ref{Interpolation_properties_1}) that
\begin{align*}
\lim_{\delta \rightarrow 0}\lim_{\tau,h \rightarrow 0}\mathrm{III} \leq
 \lim_{\delta \rightarrow 0}\lim_{\tau,h \rightarrow 0} C \left(C_n h   + C_n^{1/2}\delta^{-1/2}h + C_n\delta + \nos{g_n-\Pi_h g_{n}}\right) =0.
\end{align*}
By Lemma \ref{Lemma_Difference-num.Schemes} the fourth term satisfies
\begin{align*}
\lim_{n \rightarrow \infty} \mathrm{IV}\leq 
\lim_{n \rightarrow \infty} \max_{i=1,\ldots,N}\E{\nos{\Xi-\xii}}
%\lim_{n \rightarrow \infty}\E{\nos{\Xch-\overline{X}_{\tau,h}^{\eps}}}
\leq  \lim_{n \rightarrow \infty} C\left(\E{\nos{x_0-x^n_0}}+\nos{g-g_n} \right )= 0.
\end{align*}
Finally, we consecutively take  $\tau,h \rightarrow 0$,  $\delta \rightarrow 0$,
$n\rightarrow \infty$ and $\eps \rightarrow 0$ in (\ref{err_ineq})
and use the above convergence of $\mathrm{I-IV}$ to obtain (\ref{numconv}).
\end{proof}

\bibliographystyle{plain}
\bibliography{refs_short}

\begin{thebibliography}{1}

\bibitem{book_attouch}
H.~Attouch, G.~Buttazzo, and G.~Michaille.
\newblock {\em Variational analysis in {S}obolev and {BV} spaces}, volume~6 of
  {\em MPS/SIAM Series on Optimization}.
\newblock Society for Industrial and Applied Mathematics (SIAM), Philadelphia,
  PA; Mathematical Programming Society (MPS), Philadelphia, PA, 2006.
\newblock Applications to PDEs and optimization.

\bibitem{Roeckner_TVF_paper}
V.~Barbu and M.~R\"{o}ckner.
\newblock Stochastic variational inequalities and applications to the total
  variation flow perturbed by linear multiplicative noise.
\newblock {\em Arch. Ration. Mech. Anal.}, 209(3):797--834, 2013.

\bibitem{stvf_weak}
\v{L}. Ba\v{n}as and M.~Ondrej\'{a}t.
\newblock Numerical approximation of probabilistically weak and strong
  solutions of the stochastic total variation flow.
\newblock {\em M2AN Math. Model. Numer. Anal.}, 2022.
\newblock accepted.

\bibitem{our_paper}
\v{L}. Ba\v{n}as, M.~R\"{o}ckner, and A.~Wilke.
\newblock Convergent numerical approximation of the stochastic total variation
  flow.
\newblock {\em Stoch. Partial Differ. Equ. Anal. Comput.}, 9(2):437--471, 2021.

\bibitem{stvf_erratum}
\v{L}. Ba\v{n}as, M.~R\"{o}ckner, and A.~Wilke.
\newblock Correction to: Convergent numerical approximation of the stochastic
  total variation flow.
\newblock {\em Stoch. Partial Differ. Equ. Anal. Comput.}, 2022.
\newblock accepted.

\bibitem{BrennerS02}
S.~C. Brenner and L.~R. Scott.
\newblock {\em The Mathematical Theory of Finite Element Methods (second
  edition)}.
\newblock Springer-Verlag, New York, 2002.

\bibitem{Prohl_TVF_numerics}
X.~Feng and A.~Prohl.
\newblock Analysis of total variation flow and its finite element
  approximations.
\newblock {\em M2AN Math. Model. Numer. Anal.}, 37(3):533--556, 2003.

\bibitem{Roeckner_book}
W.~Liu and M.~R\"{o}ckner.
\newblock {\em Stochastic partial differential equations: an introduction}.
\newblock Universitext. Springer, Cham, 2015.

\end{thebibliography}

\end{document}